\documentclass[a4paper,3p,sort&compress,final]{elsarticle}
\usepackage[utf8]{inputenc}
\usepackage[english]{babel}

\usepackage{amsmath,amssymb,amsfonts,amsthm}
\usepackage{textcomp}
\usepackage{graphics}
\usepackage{xcolor}
\usepackage[ruled,linesnumbered, inoutnumbered]{algorithm2e}
\usepackage{tikz}
\usetikzlibrary{backgrounds,patterns,arrows,positioning,fit,shapes.geometric,decorations.pathmorphing}

\newtheorem{theorem}{Theorem}
\newtheorem{lemma}{Lemma}
\newtheorem{proposition}[theorem]{Proposition}
\newtheorem{remark}{Remark}
\newtheorem{definition}{Definition}

\newcommand{\calZ}[1]{\mathcal{#1}}

\newcommand{\set}[1]{\left\{#1\right\}}
\newcommand{\cset}[2]{\left\{#1\mid #2\right\}}

\newcommand{\Norm}[1]{\lVert #1 \rVert}

\newcommand{\id}{\mathrm{id}}

\newcommand{\Var}[1]{\mathcal{#1}}

\newcommand{\tensor}[1]{\mathfrak{#1}}
\newcommand{\vect}[1]{\mathbf{#1}}
\newcommand{\sten}[3]{\vect{#1}_{#2}^{#3}}
\newcommand{\Tang}[2]{\mathrm{T}_{#1} {#2}}

\newcommand{\R}{\mathbb{R}}
\newcommand{\deriv}[2]{\mathrm{d}_{#2}#1}

\newcommand{\refthm}[1]{{Theorem \ref{#1}}}
\newcommand{\reflem}[1]{{Lemma \ref{#1}}}
\newcommand{\refeqn}[1]{{(\ref{#1})}}

\newcommand{\refsec}[1]{{Section \ref{#1}}}

\newcommand{\reffig}[1]{{Figure \ref{#1}}}

\newcommand{\refprop}[1]{{Proposition \ref{#1}}}

\begin{document}

\title{Convergence analysis of Riemannian Gauss--Newton methods and its connection with the geometric condition number}

\author{Paul Breiding\fnref{cor2}}
\ead{breiding@mis.mpg.de}
\address{Max-Planck-Institute for Mathematics in the Sciences, Leipzig, Germany.}
\fntext[cor2]{Funding: The author was partially supported by DFG research grant BU 1371/2-2.}

\author{Nick Vannieuwenhoven\fnref{cor1}}
\ead{nick.vannieuwenhoven@cs.kuleuven.be}
\address{KU Leuven, Department of Computer Science, Leuven, Belgium.}
\fntext[cor1]{Funding: The author was supported by a Postdoctoral Fellowship of the Research Foundation---Flanders (FWO).}

\begin{abstract}
We obtain estimates of the multiplicative constants appearing in local convergence results of the Riemannian Gauss--Newton method for least squares problems on manifolds and relate them to the geometric condition number of [P. B\"urgisser and F. Cucker, Condition: The Geometry of Numerical Algorithms, 2013].

\end{abstract}

\maketitle
\allowdisplaybreaks

\section{Introduction}
Many problems in science and engineering are \textit{parameter identification problems} (PIPs). Herein, there is a parameter domain $\Var{M} \subset \R^M$ and a function
\(
 \Phi : \Var{M} \to \R^N.
\)
Given a point $\vect{y}$ in the image of $\Phi$, the PIP asks to identify parameters $\vect{x} \in \Var{M}$ such that $\vect{y} = \Phi(\vect{x})$; note that there could be several such parameters. For example, computing $QR$, $LU$, Cholesky, polar, singular value and eigendecompositions of a given matrix $A \in \R^{m \times n} \simeq \R^{mn}$ are examples of this. In other cases we have a tensor $\tensor{A} \in \R^{m_1 \times \cdots \times m_d}$ and need to compute CP, Tucker, block term, hierarchical Tucker, or tensor trains decompositions \cite{GKT2013}.

If the object $\widetilde{\vect{y}} \in \R^N$ whose parameters should be identified originates from applications, then usually $\widetilde{\vect{y}} \not\in \Phi(\Var{M})$. Nevertheless, in this setting one seeks parameters $\vect{x} \in \Var{M}$ such that $\vect{y} := \Phi(\vect{x})$ is as close as possible to $\widetilde{\vect{y}}$, e.g., in the Euclidean norm. This can be formulated as a nonlinear least squares problem:
\begin{align}\label{eqn_riemannian_optim_problem}
 \widetilde{\vect{y}} \mapsto \underset{\vect{x} \in \Var{M}}{\arg\min} \; \tfrac{1}{2} \Norm{ \Phi(\vect{x}) - \widetilde{\vect{y}}}^2.
\end{align}

Here, we deal with functions $\Phi$ that offer differentiability guarantees, so that continuous optimization methods can be employed for solving \refeqn{eqn_riemannian_optim_problem}. Specifically, we assume that $\Var{M}$ is a smooth embedded submanifold\footnote{{Both the optimization problem \refeqn{eqn_riemannian_optim_problem} and the condition number of maps between manifolds can be defined for abstract manifolds. Nevertheless, we consider embedded manifolds because it greatly simplifies the proof of the main theorem, allowing us to compare tangent spaces in the ambient space using Wedin's theorem \cite[Chapter III, Theorem 3.9]{MPT}. This is no longer possible for abstract manifolds, which would make the letter much more difficult to understand. In practice, many manifolds are naturally embedded.}} of $\R^M$ and that $\Phi$ is a smooth function on $\Var{M}$ \cite[Chapters 1 and 2]{Lee2013}. Hence, \refeqn{eqn_riemannian_optim_problem} is a \textit{Riemannian optimization problem} that can be solved using, e.g., Riemannian Gauss--Newton (RGN) methods \cite{AMS2008}; see \refsec{sec_riemannian_optimization}.

The sensitivity of $\vect{x} \in \Var{M}$ with respect to perturbations of $\vect{y} = \Phi(\vect{x})$ might impact the performance of these RGN methods. Let $\Psi : \Var{X} \to \Var{Y}$ be a smooth map between manifolds $\Var{X}$ and $\Var{Y}$, and let $\Tang{\vect{x}}{\Var{X}}$ denote the tangent space to the manifold $\Var{X}$ at $\vect{x}\in\Var{X}$. We recall from \cite[Section 14.3]{BC2013} that the geometric condition number $\kappa(\vect{x})$ characterizes to first-order the sensitivity of the output $\vect{y} = \Psi(\vect{x})$ to input perturbations as the spectral norm of the derivative operator $\deriv{\Psi}{\vect{x}} : \Tang{\vect{x}}{\Var{X}} \to \Tang{\Psi(\vect{x})}{\Var{Y}}$; that is, $\kappa(\vect{x}) := \| \deriv{\Psi}{\vect{x}} \| := \max_{\vect{t}\in\Tang{\vect{x}}{\Var{X}}} \| \deriv{\Psi}{\vect{x}}(\vect{t}/\|\vect{t}\|) \|$. In the case of PIPs, the geometric condition number is derived as follows.
Assume that there exists an open neighborhood $\Var{N}$ of $\vect{x} \in \Var{M}$ such that $\Var{M} = \Phi(\Var{N})$ is a smooth manifold with $m = \dim \Var{M} = \dim \Var{N}$. Since $\Phi|_\Var{N} : \Var{N} \to \Var{M}$ is a smooth map between manifolds, the inverse function theorem for manifolds \cite[Theorem 4.5]{Lee2013} entails that there exists a unique inverse function $\Phi_{\vect{x}}^{-1}$ whose derivative satisfies $\deriv{\Phi_{\vect{x}}^{-1}}{\Phi(\vect{x})} = (\deriv{\Phi}{\vect{x}})^{-1}$, provided that $\deriv{\Phi}{\vect{x}}$ is injective. Hence, the geometric condition number of the \emph{parameters}\footnote{Note that this is the geometric condition number at the output rather than the input of $\Phi_{\vect{x}}^{-1}$. The reason is that the PIP can have several $\vect{x}_i\in\Var{M}$ as solutions. Since the RGNs will only output one of these solutions, say $\vect{x}_1$, the natural question is whether this computed solution $\vect{x}_1$ is stable to perturbations of $\Phi(\vect{x}_1)$.} $\vect{x}$ is
\begin{align}\label{eqn_condition_number}
\kappa(\vect{x}) := \| \deriv{\Phi_{\vect{x}}^{-1}}{\Phi(\vect{x})} \| = \| (\deriv{\Phi}{\vect{x}})^{-1} \| = \frac{1}{\varsigma_m( \deriv{\Phi}{\vect{x}}) )},
\end{align}
where $\varsigma_m(A)$ is the $m$th largest singular value of the linear operator $A$. If the derivative is not injective, then the condition number is defined to be $\infty$.

In this letter, we show that the condition number of the parameters $\vect{x}$ in \refeqn{eqn_condition_number} appears naturally in the multiplicative constants in convergence estimates of RGN methods. Our main contribution is Theorem~\ref{thm_rgn_convergence}.

\section{The Riemannian Gauss--Newton method} \label{sec_riemannian_optimization}
Recall that a Riemannian manifold $(\Var{M},\langle \cdot , \cdot \rangle)$ is a smooth manifold $\Var{M}$, where for each ${p}\in \calZ{M}$ the tangent space $\Tang{p}{\Var{M}}$ is equipped with an inner product $\langle \cdot , \cdot \rangle_{p}$ that varies smoothly with ${p}$; see \cite[Chapter~13]{Lee2013}. The zero element of $\Tang{p}{\Var{M}}$ is denoted by $0_p$. Since we deal exclusively with embedded submanifolds $\Var{M} \subset \R^M$, we take $\langle \vect{a}, \vect{b} \rangle_p := \vect{a}^T \vect{b}$ equal to the standard inner product on $\R^M$. In the following we drop the subscript~``$p$.'' The induced norm is $\Norm{\vect{v}} = \sqrt{\langle \vect{v},\vect{v} \rangle}$.
% When there can be no confusion we also omit the index ${p}$ in both $\langle \cdot ,\cdot \rangle_{p}$ and $\Norm{\cdot}_{p}$.
The \emph{tangent bundle} of a manifold $\Var{M}$ is the smooth vector bundle $\Var{T}\Var{M} := \cset{ (p, \vect{v}) }{p \in \Var{M}, \vect{v} \in \Tang{p}{\Var{M}}}$.

In the remainder of this letter, we let $\Var{M} \subset \R^M$ be an embedded submanifold with $m = \dim \Var{M} \le M$ equipped with the standard Riemannian metric inherited from $\R^M$.
Riemannian optimization methods can be applied to the minimization of a least-squares cost function
\begin{equation}\label{least_squares}
 f : \Var{M} \to \R,\;  p \mapsto \tfrac{1}{2} \| F(p) \|^2 \quad\text{with } F : \Var{M} \to \R^N.
\end{equation}
Recall that Newton's method for minimizing $f$ consists of choosing a $x_0 \in \Var{M}$ and then generating a sequence of iterates $x_1$, $x_2$, $\ldots$ in $\Var{M}$ according to the following process:
\begin{align}\label{eqn_newton_process}
 x_{k+1} \leftarrow R_{x_{k}}(\eta_{k}) \quad\text{with } \bigl( \nabla_{x_k}^2 f \bigr) \eta_{k} = -\nabla_{x_k} f;
\end{align}
herein, $\nabla_{x_{k}} f : \Tang{x_k}{\Var{M}} \to \R$ is the \textit{Riemannian gradient}, and $\nabla_{x_{k}}^2 f : \Tang{x_k}{\Var{M}} \to \Tang{x_k}{\Var{M}}$ is the \textit{Riemannian Hessian}; for details see \cite[Chapter 6]{AMS2008}. The map $R_{x_k} : \Tang{x_k}{\Var{M}} \to \Var{M}$ is a \textit{retraction operator}.
% Recall the following definition of a retraction from \cite{AMS2008,KSV2014}.

\begin{definition}[Retraction \cite{AMS2008,KSV2014}] \label{def_retraction}
A \emph{retraction} $R$ is a map from an open subset $\Var{T}\Var{M} \supset \Var{U} \to \Var{M}$ that satisfies all of the following properties for every $p \in \Var{M}$:
\begin{enumerate}
\item $R(p,0_p) = p$;
\item $\Var{U}$ contains a neighborhood $\Var{N}$ of $(p,0_p)$ such that the restriction $R|_\Var{N}$ is smooth;
\item $R$ satisfies the local rigidity condition $\deriv{R(x,\cdot)}{0_x} = \id_{\Tang{x}{\Var{M}}}$ for all $(x, 0_x) \in \Var{N}$.
\end{enumerate}
We let $R_p(\cdot) := R(p,\cdot)$ be the retraction $R$ with foot at $p$.
\end{definition}

% Every manifold has at least one retraction, namely the exponential map \cite{AMS2008}.
A retraction is a first-order approximation of the exponential map \cite{AMS2008}; the following result is well-known.
% The following is well-known.
\begin{lemma}\label{lem_retraction}
Let $R$ be a retraction. Then for all $x\in\Var{M}$ there exists some $\delta_x >0$ such that for all $\eta\in \Tang{x}{\Var{M}}$ with $\Norm{\eta}<\delta_x$ one has $R_x(\eta) = x + \eta + \mathcal{O}(\Norm{\eta}^2).$
\end{lemma}

As stated in \cite[Section 8.4.1]{AMS2008}, the RGN method for minimizing $f$ is obtained by replacing the $\nabla_{x_{k}}^2 f$ in the Newton process \refeqn{eqn_newton_process} by the Gauss--Newton approximation $(\deriv{F}{x_k})^* \circ (\deriv{F}{x_k})$; herein $A^*$ denotes the adjoint of the bounded linear operator $A$ with respect to the inner product $\langle \cdot, \cdot \rangle$.
Note that an explicit expression for the update direction $\eta_k$ can be obtained. The Riemannian gradient is
\begin{align} \label{eqn_gradient}
\nabla_{x_{k}} f
= \nabla_{x_{k}} \tfrac{1}{2} \langle F(x), F(x) \rangle
% = \langle F(x_{k}), \deriv{F}{x_{k}} \rangle
= (\deriv{F}{x_k})^* \bigl( F(x_k) \bigr);
\end{align}
see \cite[Section 8.4.1]{AMS2008}.
If $\deriv{F}{x_k}$ is injective, then the solution of the system in \refeqn{eqn_newton_process} with the Riemannian Hessian replaced by the Gauss--Newton approximation is given explicitly by
\[
\eta_k = -\bigl( (\deriv{F}{x_k})^* \circ (\deriv{F}{x_k}) \bigr)^{-1} (\deriv{F}{x_k})^* \bigl( F(x_k) \bigr) =: -(\deriv{F}{x_k})^\dagger \bigl( F(x_k) \bigr).
\]

\section{Main result: Convergence analysis of the RGN method}
We prove in this section that both the convergence rate and radius of the RGN method are influenced by the condition number of the PIP at the local minimizer. In the case of PIPs, we have $F(\vect{x}) := \Phi(\vect{x}) - \widetilde{\vect{y}}$ for some fixed $\widetilde{\vect{y}} \in \R^N$. Hence, $\deriv{F}{\vect{x}} = \deriv{\Phi}{\vect{x}}$, so that the next theorem relates the geometric condition number~\refeqn{eqn_condition_number} to the convergence properties of the RGN method for solving the least-squares problem \refeqn{least_squares}.

\begin{remark}
The RGN method is only locally convergent. Practical methods are obtained by adding a globalization strategy \cite{NW2006,AMS2008} such as a line search or trust region scheme. The goal of these strategies is guaranteeing sufficient descent for global convergence, while preserving the local rate of convergence. In the main theorem, we present the analysis without globalization strategy, so as to focus on the main idea of the proof. In case of a trust region scheme, the usual approach for extending the proof consists of showing that close to a local minimizer, the unconstrained Newton step is always contained in the trust region and hence selected. This will be true if the starting point is sufficiently close to the local minimizer. Then, the local rate of convergence will be the same as when no trust region scheme is employed.
\end{remark}

In the remainder of this section, let $B_{\tau}(\vect{x})$ denote the ball of radius $\tau$ centered at $\vect{x} \in \R^M$. The following is the main theorem of this letter.

\begin{theorem}\label{thm_rgn_convergence}
% Let $\Var{M}$ be an embedded submanifold of $\R^{M}$ with $m = \dim \Var{M} \le M$ equipped with the standard Riemannian metric inherited from $\R^M$.
Assume that $\vect{x}_\star \in \Var{M}$ is a local minimum of the objective function $f$ from \refeqn{least_squares}, where $\deriv{F}{\vect{x}_\star}$ is injective. Let $\kappa := ( \varsigma_m( \deriv{F}{\vect{x}_\star} ) )^{-1} > 0$.
Then, there exists $\epsilon'>0$ such that for all $0<\alpha<1$ there exists a universal constant $c>0$ depending on $\epsilon'$, $F$, $\vect{x}_\star$, $\Var{M}$, and $R$ so that the following holds.
\begin{enumerate}
\item {(Linear convergence):} If $\frac{c\kappa^2 \Vert F(\vect{x}_\star)\Vert}{\alpha} < 1,$ then for all $\vect{x}_0 \in B_{\epsilon}(\vect{x}_\star) \cap \Var{M}$
with
\[
  \epsilon := \min\Bigl\{ \frac{1-\alpha}{c\kappa}, \frac{\alpha\epsilon'}{1+\alpha+c\kappa^2\Vert F(\vect{x}_\star)\Vert}\Bigr\},
\]
the RGN method generates a sequence $\vect{x}_0, \vect{x}_1, \ldots$ that converges linearly to $\vect{x}_\star$. In fact,
\[
 \| \vect{x}_\star - \vect{x}_{k+1} \| \le \frac{c\kappa^2 \Vert F(\vect{x}_\star)\Vert}{\alpha} \| \vect{x}_\star - \vect{x}_{k} \| + \Var{O}(\| \vect{x}_\star - \vect{x}_{k} \|^2).
\]
\item {(Quadratic convergence):} If $\vect{x}_\star$ is a zero of the objective function $f$, then for all $\vect{x}_0 \in B_{\epsilon}(\vect{x}_\star) \cap \Var{M}$ with
\[
  \epsilon := \min\Bigl\{\frac{1-\alpha}{c\kappa}, \frac{\alpha\epsilon'}{1+\alpha} \Bigr\},
\]
the RGN method generates a sequence $\vect{x}_0, \vect{x}_1, \ldots$ that converges quadratically to $\vect{x}_\star$. In fact,
\[
 \|  \vect{x}_\star -  \vect{x}_{k+1}\| \le \frac{c(\kappa+1)}{\alpha} \|  \vect{x}_\star -  \vect{x}_{k}\|^2 + \mathcal{O}( \| \vect{x}_\star -  \vect{x}_k\|^3 ).
\]
\end{enumerate}
\end{theorem}

\begin{remark}
The {order of convergence} may also be established from \cite[Theorem 8.2.1]{AMS2008}. However, {intrinsic} multiplicative constants are not derived there, as their analysis is founded on {coordinate expressions}
 that depend on the chosen chart; they thus only derive chart-dependent multiplicative constants.
\end{remark}

In the following let $P_A$ denote the orthogonal projection onto the linear subspace $A \subset \R^N$.  Recall the next lemma from \cite[Section 2]{BV2017}, which we need in the proof of \refthm{thm_rgn_convergence}.

\begin{lemma} \label{lem_kind_of_taylor_series}
Let $F : \Var{M} \to \R^N$ be a smooth function and $\vect{x} \in \Var{M}$. Then, there exist constants $r_F > 0$ and $\gamma_F \ge 0$ such that for all $\vect{y} \in B_{r_F}(\vect{x}) \cap \Var{M}$ we have
\(
F(\vect{y}) = F(\vect{x}) + (\deriv{F}{\vect{x}}) P_{\Tang{\vect{x}}{\Var{M}}} \Delta + \vect{v}_{\vect{x},\vect{y}},
\)
where $\Delta =$ $(\vect{y}-\vect{x}) \in \R^N$ and
\(
\|\vect{v}_{\vect{x},\vect{y}}\| \le \gamma_F \|\Delta\|^2.
\)
\end{lemma}

We can now prove \refthm{thm_rgn_convergence}.

\begin{proof}[Proof of \refthm{thm_rgn_convergence}]
We begin with some general considerations: In \reflem{lem_retraction} we choose $\delta$ small enough such that it applies to \emph{all}  $\vect{x}\in B_{\delta}(\vect{x}_\star) \cap \Var{M}$. Let $0<\epsilon'\leq\delta$. Then, there exists a constant $\gamma_R>0$ depending on the retraction operator $R$, such that for all $\vect{x}\in B_{\epsilon'}(\vect{x}_\star) \cap \Var{M}$ we have
\begin{equation}
\label{e0} \| R_{\vect{x}}(\eta) - (\vect{x} + \eta) \| \le \gamma_R \| \eta \|^2  \;\text{ for every } \eta \in B_{\epsilon'}(0) \cap \Tang{\vect{x}}{\Var{M}}.
\end{equation}
By applying \reflem{lem_kind_of_taylor_series} to the smooth functions $F$ and $\id_{\Var{M}}$ respectively and using the smoothness of (the derivative of) $F$, we see that there exist constants  $\gamma_F, \gamma_I >0$ so that for all $\vect{x} \in B_{\epsilon'}(\vect{x}_\star) \cap \Var{M}$ we have
\begin{align}
\label{e1} F(\vect{x}_\star) - F(\vect{x}) - (\deriv{F}{\vect{x}}) P_{\Tang{\vect{x}}{\Var{M}}} (\vect{x}_\star- \vect{x}) &= \vect{v} \text{ with } \|\vect{v} \| \le \gamma_F \| \vect{x}_\star -  \vect{x}\|^2,\\
\label{e2}  \vect{x}_\star -  \vect{x} - P_{\Tang{\vect{x}}{\Var{M}}}(  \vect{x}_\star-\vect{x}) &= \vect{w} \text{ with } \|\vect{w} \| \le \gamma_I \| \vect{x}_\star -  \vect{x}\|^2.
\end{align}
Moreover, we define the Lipschitz constant
\begin{align}\label{def_C}
C &:= \max\limits_{\vect{x}\in B_{\epsilon'}(\vect{x}_\star) \cap \Var{M}}\, \frac{\Norm{\deriv{F}{\vect{x}_\star} \circ  P_{\Tang{\vect{x}_\star}{\Var{M}}} - \deriv{F}{\vect{x}} \circ P_{\Tang{\vect{x}}{\Var{M}}} }}{\Norm{\vect{x}_\star-\vect{x}}}.
\end{align}
We choose a constant $c$, depending on $R,F$ and $\epsilon'$, that satisfies
\begin{align} \label{eqn_crazy_bound}
  \epsilon \leq \min\Bigl\{ \frac{1}{\kappa \gamma_F}, \frac{1-\alpha}{ C \kappa}, \Big( 1 + \frac{1+\tfrac{1}{2}(1+\sqrt{5}) C \kappa^2  \|F(\vect{x}_\star) \|}{\alpha} \Big)^{-1} \epsilon'\Bigr\}
\text{ and }
  c\geq \max \{ \tfrac{1}{2}(1+\sqrt{5})C, \gamma_F,\gamma_R,\gamma_I \}.
  \end{align}

The rest of the proof is by induction. Suppose that the RGN method applied to starting point $\vect x_0\in\Var M$ generated the sequence of points $\vect x_0,\ldots, \vect{x}_k \in \Var{M}$. First, we show that~$\deriv{F}{\vect x_k}$ is injective, so that the update direction $\eta = -(\deriv{F}{\vect x_k})^\dagger F(\vect{x_k})$, and, hence, $\vect{x}_{k+1}$ is defined. Thereafter, we prove the asserted bounds on $\Vert \vect x_k -  \vect{x}_{k+1}\Vert$.
For avoiding subscripts, let $\vect{x} := \vect{x}_k \in \Var{M}$ and $\vect{y} := \vect{x}_{k+1} = R_{\vect{x}}( -(\deriv{F}{\vect{x}})^\dagger F(\vect{x}) ) \in \Var{M}$.

{By induction,} we can assume $\Vert \vect x_k - \vect x_\star\Vert \leq \Vert \vect x_0 - \vect x_\star\Vert$; {indeed the base case $k=0$ is trivially true, and we will prove that it also holds for $k+1$ at the end of this proof, completing the induction.}
Hence,
% \begin{equation}\label{x_epsilon}
\(
\vect x\in B_{\epsilon'}(\vect x_\star)\cap \Var M.
\)
% \end{equation}
Let $J \in \R^{N \times M}$ denote the matrix of $\deriv{F}{\vect{x}}$ with respect to the standard bases on $\R^M$ and $\R^N$. Let $J = U \Sigma V^T$ be its (compact) singular value decomposition (SVD), where $U \in \R^{N \times m}$ and $V \in \R^{M \times m}$ have orthonormal columns, the columns of $V$ span $\Tang{\vect{x}}{\Var{M}}$ and $\Sigma\in\R^{m \times m}$ is diagonal matrix containing the singular values. Then, the matrix of $(\deriv{F}{\vect{x}})^\dagger$ with respect to the standard bases is $J^\dagger$, i.e., the Moore--Penrose pseudoinverse of $J$, and $\kappa(\vect{x})=\Vert J^\dagger \Vert$. Similarly, let $J_{\star} \in \R^{N \times M}$ denote the matrix of $\deriv{F}{\vect{x}_\star}$, and let $U_\star \Sigma_\star V_\star^T$ be its SVD.

By assumption, $\deriv{F}{\vect x_\star}$ is injective and thus we have $\kappa^{-1}=\varsigma_{\min}(\deriv{F}{\vect{x}_\star})=\varsigma_{\min}(J_{\star}) > 0.$ The matrix of $P_{\Tang{\vect{x}}{\Var{M}}}$ is $VV^T$, and similarly for $P_{\Tang{\vect{x}_\star}{\Var{M}}}$.
Then, by the definition of $C$ in \refeqn{def_C}, we have
\begin{equation}\label{eqn_stuff_deriv_bound}
 \| J_\star - J \| = \| J_\star (V_\star V_\star^T) - J (V V^T) \| \le C \|  \vect{x}_\star -  \vect{x} \|,
\end{equation}
and hence $\| J_\star - J \|\le C\epsilon \le C \tfrac{1-\alpha}{c \kappa} \le (1-\alpha) \varsigma_{\min}(J_\star)$, because $\vect{x} \in B_{\epsilon}(\vect{x}_\star)$ and the definition of $c$.
From Weyl's perturbation Lemma it follows that
$ | \varsigma_{\min}(J_\star) - \varsigma_{\min}(J) | \le \| J_\star - J \| \le (1-\alpha) \varsigma_{\min}(J_\star).$ We obtain $\varsigma_{\min}(J) > \alpha \varsigma_{\min}(J_\star) > 0$, where the last inequality is by the assumption $\alpha>0$. It follows that
\begin{align} \label{eqn_stuff_cond_bound}
\|J^\dagger\| = \bigl( \varsigma_{\min}(J) \bigr)^{-1} < \kappa \alpha^{-1} < \infty,
\end{align}
so that $\deriv{F}{\vect{x}}$ is indeed injective. This shows that the RGN update direction $\eta$ is well defined.

It remains to prove the bound on $\Norm{\vect{x}_\star-\vect{y}}$.
First we show that $\|\eta\| = \|-J^\dagger F(\vect{x})\| \le \epsilon' < \delta$, so that the retraction would satisfy \refeqn{e0}.
By assumption $\vect{x}_\star$ is a local minimum of \refeqn{least_squares}, so that from \refeqn{eqn_gradient} we obtain $0 = \nabla_{\vect{x}_\star} f = J_\star^T F(\vect{x}_\star) =  V_\star \Sigma_\star U_\star^T F(\vect{x}_\star)$. By \cite[Chapter III, Theorem 1.2 (9)]{MPT} and the assumption that $\deriv{F}{\vect{x}_\star}$ is injective, we have $J_\star^\dagger = V_\star \Sigma_\star^{-1} U_\star^T$ from which we conclude $J_\star^\dagger  F(\vect{x}_\star)=0$.
Let $P = P_{\Tang{\vect{x}}{\Var{M}}}$. From \refeqn{e1},
\begin{equation}\label{some_of_the_last}
J^\dagger F(\vect{x}) = J^\dagger F(\vect{x}_\star)   -   P ( \vect{x}_\star -  \vect{x}) - J^\dagger  \vect{v},
\end{equation}
so that
\begin{align}
\Norm{\eta} = \| -(J^\dagger-J_\star^\dagger) F(\vect{x}_\star)   +   P ( \vect{x}_\star -  \vect{x}) + J^\dagger  \vect{v}\|
\leq \| J^\dagger-J_\star^\dagger\| \|F(\vect{x}_\star) \|  +   \|P\| \| \vect{x}_\star -  \vect{x}\| + \|J^\dagger\|  \|\vect{v}\|.
\label{eqn_stuff_bounds9}
\end{align}
From Wedin's theorem \cite[Chapter III, Theorem 3.9]{MPT} we obtain
	\begin{equation}\label{wedin_appl}
	\| J^\dagger-J_\star^\dagger\| \leq \frac{1+\sqrt{5}}{2} \, \| J^\dagger\| \|J_\star^\dagger\| \| J-J_\star\| \leq \frac{(1+\sqrt{5}) \, C\kappa^2}{2 \alpha}  \|  \vect{x}_\star -  \vect{x} \|,
	\end{equation}
where the last step is because of \refeqn{eqn_stuff_deriv_bound} and \refeqn{eqn_stuff_cond_bound}.
Using $\|P\|=1$ for orthogonal projectors, the assumption $\| \vect{x}_\star - \vect{x} \| \le (\kappa \gamma_F)^{-1}$, \refeqn{eqn_stuff_cond_bound}, \refeqn{wedin_appl}, and the bound on $\Norm{\vect{v}}$ in \refeqn{e1}, it follows from \refeqn{eqn_stuff_bounds9} that
\begin{equation}\label{zeta_2_bound1}
\Norm{\eta} \le \Big( 1 + \frac{\kappa\gamma_F\Norm{\vect{x}_\star-\vect{x}}+\tfrac{1}{2}(1+\sqrt{5}) C \kappa^2  \|F(\vect{x}_\star) \|}{\alpha} \Big) \| \vect{x}_\star -  \vect{x}\|.
\end{equation}
By the definition of $\epsilon$ and the assumption $\Norm{\vect{x}_\star-\vect{x}}<\epsilon$, we have \( \Norm{\vect{x}_\star-\vect{x}}<\epsilon < \frac{1}{\kappa \gamma_F},\)
so that by \refeqn{zeta_2_bound1},
\begin{equation}
  \label{zeta_2_bound}
\Norm{\eta} \le \Big( 1 + \frac{1+\tfrac{1}{2}(1+\sqrt{5}) C \kappa^2  \|F(\vect{x}_\star) \|}{\alpha} \Big) \| \vect{x}_\star -  \vect{x}\|.
\end{equation}
Using the third bound on $\epsilon$ in \refeqn{eqn_crazy_bound}, we have
$$
\Norm{\vect{x}_\star-\vect{x}}<\epsilon <\Big( 1 + \frac{1+\tfrac{1}{2}(1+\sqrt{5}) C \kappa^2  \|F(\vect{x}_\star) \|}{\alpha} \Big)^{-1}\epsilon',
$$
which when plugged into \refeqn{zeta_2_bound} yields $\Norm{\eta}<\epsilon'$.

From the foregoing discussion, we conclude that \refeqn{e0} applies to $R_{\vect{x}}(\eta)=R_{\vect{x}}( - J^\dagger F(\vect{x}) )$, so that
\begin{equation}\label{eqn_stuff_bounds5}
 \|  \vect{y} -  \vect{x}_\star \| = \| R_{\vect{x}}( - J^\dagger F(\vect{x}) ) -  \vect{x}_\star \| \le \|  \vect{x} - J^\dagger F(\vect{x}) -  \vect{x}_\star \| + \gamma_R \| \eta \|^2.
\end{equation}
Let $\zeta := \|  \vect{x} - J^\dagger F(\vect{x}) -  \vect{x}_\star \|$.
We use $ J_\star^\dagger F(\vect{x}_\star)=0$ and the formula from \refeqn{some_of_the_last} to derive that
\begin{align*}
\zeta=\|\vect{x} -  \vect{x}_\star - (J^\dagger F(\vect{x})- J_\star^\dagger F(\vect{x}_\star))\|
&= \|\vect{x} -  \vect{x}_\star- (J^\dagger-J_\star^\dagger) F(\vect{x}_\star)   +   P ( \vect{x}_\star -  \vect{x}) + J^\dagger  \vect{v}\|\\
&=  \|- (J^\dagger-J_\star^\dagger) F(\vect{x}_\star)   - \vect{w} + J^\dagger  \vect{v}\|\\
&\le \| J^\dagger-J_\star^\dagger\| \|F(\vect{x}_\star)\|   +    \gamma_I \| \vect{x}_\star -  \vect{x}\|^2 + \| J^\dagger\|   \gamma_F \| \vect{x}_\star -  \vect{x}\|^2,
\end{align*}
where the second-to-last equality is due to \refeqn{e2}, and in the last line we have used the triangle inequality and the bounds on $\Norm{\vect{v}}$ and $\Norm{\vect{w}}$ from \refeqn{e1} and \refeqn{e2}.
Combining this with \refeqn{wedin_appl} and \refeqn{eqn_stuff_cond_bound} yields
\begin{align}
\label{eqn_stuff_bounds8}\zeta&\leq \frac{\tfrac{1}{2}(1+\sqrt{5}) \,C \kappa^2\, \|F(\vect{x}_\star)\|}{\alpha}  \|  \vect{x}_\star -  \vect{x} \|    +    \Big(\gamma_I  + \frac{\gamma_F\kappa}{\alpha}   \Big) \| \vect{x}_\star -  \vect{x}\|^2,
\end{align}
Note that we have chosen the constant $c$ large enough, so that $\tfrac{1}{2}(1+\sqrt{5}) \,C  < c$.
Plugging \refeqn{eqn_stuff_bounds8} and \refeqn{zeta_2_bound} into \refeqn{eqn_stuff_bounds5} yields the first bound.

For the second assertion we have the additional assumption that $\vect{x}_\star$ is a zero of the objective function $f(\vect{x})=\tfrac{1}{2}\Norm{F(\vect{x})}^2$. From \refeqn{eqn_stuff_bounds8} we obtain $\zeta\le \big(\frac{\kappa\gamma_F}{\alpha} + \gamma_I\big) \|  \vect{x}_\star -  \vect{x} \|^2.$ From \refeqn{eqn_stuff_bounds9} we get $\Norm{\eta} = \|  P ( \vect{x}_\star -  \vect{x}) + J^\dagger  \vect{v}\|$ so that we can bound $\Norm{\eta}^2$ by
\begin{align*}
\|P( \vect{x}_\star -  \vect{x})\|^2 + 2|\langle P( \vect{x}_\star- \vect{x}), J^\dagger \vect{v} \rangle| + \|J^\dagger \vect{v}\|^2
\le \| \vect{x}_\star -  \vect{x}\|^2 + 2 \gamma_F \|J^\dagger\| \| \vect{x}_\star -  \vect{x}\|^3 + \gamma_F^2 \|J^\dagger\|^2 \| \vect{x}_\star -  \vect{x}\|^4,
\end{align*}
where the inequality is by the Cauchy--Schwartz inequality and the fact that $\|P\| = 1$ for orthogonal projectors. As before, plugging these bounds for $\zeta$ and $\Norm{\eta}$ into \refeqn{eqn_stuff_bounds5} and exploiting that $c\geq \max\set{\gamma_F,\gamma_I,\gamma_R}$, the second bound is obtained.
\end{proof}

A reviewer asked how critical the injectivity assumption on the derivative $\deriv{F}{\vect{x}}$ in the above theorem is. The brief answer is that it is usually a very weak assumption in practice. First, we need a lemma.

\begin{lemma}\label{lemma_explain_injectivity}
{Let $\Var M\subset \R^M$ be an embedded manifold whose projectivization is a smooth projective variety, and let
and $\Phi:\Var M\to \R^N$ be a regular map.} Let $\Var{N}$ denote the $\R$-variety that is the Zariski closure of the image $\Phi(\Var{M})$. If the dimension are $\dim \Var M = \dim\Var{N}$, then the locus of points $\Var{G} := \{ \vect{x} \in \Var{M} \;|\; \deriv{\Phi}{\vect{x}} \text{ is injective and } \Phi(\vect{x}) \text{ is a smooth point of } \Var{N} \}$ is a dense subset of $\Var{M}$ in the Euclidean topology.
\end{lemma}
\begin{proof}
{This is essentially a restatement of \cite[Theorem 11.12]{Harris1992}.}
% Let $\vect{x}\in \Var{M}$. Then,
% $\deriv{\Phi}{\vect{x}} : \Tang{\vect x}{\Var{M}}  \to \Tang{\Phi(\vect{x})}{\Var{N}}$,
% where the tangent space on the latter is the Zariski tangent space \cite{Harris1992}. It follows from a standard result, e.g., \cite[Theorem 11.12]{Harris1992}, that if the (affine) dimensions satisfy $\dim \Var{M} =\dim \Var{N}$, then $\deriv{\Phi}{\vect{x}}$ is not injective at most on a Zariski-closed subset of $\Var{M}$, because the generic fiber is $0$-dimensional. Moreover, by generic smoothness, the singular locus of a variety is always closed \cite[p.~175--176]{Harris1992}. This entails the claim.
\end{proof}

{
The following proposition shows, under the assumptions of \reflem{lemma_explain_injectivity}, that the local optimizer $\vect x_\star$ in \refthm{thm_rgn_convergence} has an injective derivative $\deriv{F}{\vect x_\star}=\deriv{\Phi}{\vect x_\star}$ on a set of inputs (in $\R^N$) of positive measure.}

\begin{proposition}\label{prop_measure}
Let $\Var{M}$, $\Var G$ $\Var{N}$, and $\Phi$ be as in \reflem{lemma_explain_injectivity}. Assume that we have the equality $\dim \Var{M} =\dim \Var N$. Let
\(
 \Var{B} := \bigl\{ \vect{y} \in \R^N \;\big|\; \arg\min_{\vect{x}\in\Var{M}}\|\Phi(\vect{x}) - \vect{y}\| \subset \Var{G} \bigr\}
\)
be the set of points $\vect{y}$ all of whose closest approximations on $\Phi(\Var{M})$, i.e., $C_{\vect{y}} := \arg\min_{\vect{x}\in\Var{M}}\|\Phi(\vect{x}) - \vect{y}\|$, lie in $\Var{G}$. Then, $\Var{B}$ has positive Lebesgue measure, i.e., it is open in the Euclidean topology. Moreover, $\Var{N}\cap\Var{B}$ is Euclidean dense in $\Phi(\Var{M})$.
\end{proposition}
\begin{proof}
Let $\Var{W}_1 \subset \Var{M}$ be the locus where the dimension of the fiber $\Phi^{-1}(\Phi(\vect{x}))$ is strictly positive, i.e., $\Var W_1=\{\vect x \in \Var M \mid \dim \Phi^{-1}(\Phi(\vect{x})) >0\}$. By \cite[Theorem 11.12]{Harris1992} $\Var W_1$ is a Zariski-closed set. The last claim of the proposition is also a corollary of this theorem and the assumption that the generic fiber is $0$-dimensional.

It remains to show the first claim. Let $\Var{W}_2 \subset \Var{M}$ be the {Zariski-closed subset} of points $\vect{x} \in \Var{M}$ for which $\Phi(\vect{x})$ lies in the singular locus of $\Var{N}$.
Set $\Var{W} := (\Var{W}_1 \cup \overline{\Var{W}_2}) \subset \Var{M}$, where the overline denotes the closure in the Zariski topology. Note that $\Var{M}\setminus\Var{W} \subset \Var{G}$.

Let $\vect{x} \not\in \Var{W}$. Since the derivative $\deriv{\Phi}{\vect x}$ is injective, there exists a local diffeomorphism between an open neighborhood $\Var{M}_0 \subset \Var{M}$ of $\vect{x}$ and an open neighborhood $\Var{N}_0\subset \Var N$ of $\Phi(\vect{x})$. By restricting neighborhoods, we can assume that the Euclidean closure of $\Var{N}_0$ is contained in the smooth locus of $\Var{N}$ and that $\Var{M}_0$ is contained in $\Var{M}\setminus\Var{W}$. Take a tubular neighborhood $\Var{T}$ of $\Var{N}_0 \subset \R^N$ that does not intersect $\Var{N}\setminus\Var{N}_0$, and let $h$ be its height; note that $h > 0$, because the closure of $\Var{N}_0$ does not contain singular points of $\Var{N}$. Then, there exists an open ball $B_{\delta}(\Phi(\vect{x}))$ in $\R^N$ of positive radius $0 < \delta < h$, centered at $\Phi(\vect{x})$, whose intersection with $\Var{N}$ is contained in $\Var{N}_0$. By construction $B_{\delta}(\Phi(\vect{x})) \subset (\Var{N}_0 \cup \Var{T})$. It follows from the triangle inequality that the closest point on $\Var{N}$ to any point of $\Var{B}_{\vect{x}} := B_{\delta/2}(\Phi(\vect{x}))$ is contained in $\Var{N}_0 \subset \Phi(\Var{M}) \subset \Var{N}$. Since $\Var{N}_0 = \Phi(\Var{M}_0)$ and because $\Var{M}_0 \subset \Var{G}$, it follows that $\Var{B}_{\vect{x}} \subset \Var{B}$ for all $\vect{x} \in \Var{M}\setminus\Var{W}$.
\end{proof}

\section{Numerical experiments}
Here we experimentally verify the dependence of the multiplicative constant on the geometric condition number for a special case of PIP \refeqn{eqn_riemannian_optim_problem}, namely the tensor rank decomposition (TRD) problem. The model is
\[
 \Phi : \Var{S} \times \cdots \times \Var{S} \to \R^N, (\sten{a}{1}{1} \otimes \cdots \otimes \sten{a}{1}{d}, \ldots, \sten{a}{r}{1}\otimes\cdots\otimes\sten{a}{r}{d}) \mapsto \sum_{i=1}^r \sten{a}{i}{1} \otimes \cdots \otimes \sten{a}{i}{d},
\]
where $d \ge 3$, $N = m_1 \cdots m_d$, and $\Var{S} \subset \R^N$ is the manifold of $m_1 \times \cdots \times m_d$ rank-$1$ tensors \cite{Harris1992}. The image of $\Phi$ is called a \emph{join set} and the PIP is a special case of the \emph{join decomposition problem} \cite{BV2017}. To put emphasis on the join structure of the image of $\Phi$, we denote $\Var J:=\Phi(\Var{S} \times \cdots \times \Var{S})$.

In the numerical experiments of this section we apply a RGN method to $\min_{\vect{x} \in \Var{M}} \frac{1}{2} \|\Phi(\vect{x}) - \tensor{A}\|^2$, where $\Var{M} := \Var{S} \times \cdots \times \Var{S}$ is the $r$-fold product manifold of $\Var{S}$, and $\tensor{A} \in \R^N$ is the given tensor to approximate. We choose the retraction operator $R : \Var{T}\Var{M} \to \Var{M}$ from \cite{BV2017b}.

The projectivization of the manifold $\Var{S}$ is called the Segre variety; it is a smooth, irreducible projective variety with affine dimension $\dim \Var{S} = 1 + \sum_{k=1}^d (m_k-1)$. The problem of computing the dimension of the Zariski-closure $\overline{\Var{J}}$, which is called the $r$-secant variety of $\Var{S}$, has been classically studied; see \cite[Section 5.5]{Landsberg2012} for an overview. The results of \cite{COV2014} entail that the dimension equality {$\dim \Var{M} = \dim \overline{\Var{J}}$} is satisfied for all $r \cdot \dim \Var{S} < N$ and $N \le 15000$, subject to a few theoretically characterized exceptions. In the example below, {we take $r=2$ for which the dimension equality is always satisfied \cite{AOP2009}.} Hence, \refprop{prop_measure} entails that the injectivity assumption in \refthm{thm_rgn_convergence} is satisfied at least on a set of positive Lebesgue measure. Therefore, the convergence rate of the RGN method is influenced by the geometric condition number of the optimal parameters $\vect{x}_\star \in \Var{M}$ that minimizes the objective function.

We showed in \cite[Section~5.1]{BV2017} that the condition number of the above PIP at $\vect{x} = (\sten{a}{i}{1}\otimes\cdots\otimes\sten{a}{i}{d})_{i=1}^r$ is $\kappa(\vect{x}) = \bigl( \varsigma_{m}(U) \bigr)^{-1}$,  where $m = r \cdot \dim \Var{S}$, and the matrix $U\in \R^{N\times m}$ is given by $U = \left[\begin{smallmatrix} U_1 & \cdots & U_r \end{smallmatrix}\right]$ with
\[
U_i :=
\begin{bmatrix}
\tfrac{\sten{a}{i}{1}}{\|\sten{a}{i}{1}\|} \otimes \cdots \otimes \tfrac{\sten{a}{i}{d}}{\|\sten{a}{i}{d}\|}
& Q_{1,i} \otimes \tfrac{\sten{a}{i}{2}}{\|\sten{a}{i}{2}\|} \otimes \cdots \otimes \tfrac{\sten{a}{i}{d}}{\|\sten{a}{i}{d}\|}
& \cdots
& \tfrac{\sten{a}{i}{1}}{\|\sten{a}{i}{1}\|} \otimes \cdots \otimes \tfrac{\sten{a}{i}{d-1}}{\|\sten{a}{i}{d-1}\|} \otimes Q_{d,i}
\end{bmatrix},
\]
where $Q_{k,i} \in \R^{m_k \times (m_k-1)}$ is a matrix containing an orthonormal basis of the orthogonal complement of $\sten{a}{i}{k}$ in~$\R^{m_k}$. These expressions allow us to compute the condition number at any given decomposition $\vect{x} \in \Var{M}$.%

All of the following computations were performed in Matlab R2016b. For clearly illustrating the rates of convergence, we used variable precision arithmetic (vpa) with $400$ digits of accuracy. Since performing experiments in vpa is very expensive, we consider only the tiny example of a rank-$2$ tensor of size $3 \times 3 \times 3$. We showed in \cite{BV2017b} that an implementation of the RGN method with trust region globalization strategy applied to the above PIP formulation, can outperform state-of-the-art optimization methods for the tensor rank approximation problem on small-scale, dense problems with $r \sum_{k=1}^d m_k \lesssim 1000$.

\subsection{Experiment 1: Random perturbations}
Consider the following parametrized tensors in $\R^3\otimes \R^3\otimes \R^3$. For $s\geq 0$ we let $\vect{x}(s) = (x(s),\vect{e}_2^{\otimes 3})\in \Var S \times \Var S$ where $x(0) := \vect{e}_1^{\otimes 3}$ and $x(s) := (\vect{e}_2 - 2^{-s} \vect{e}_1)^{\otimes 3}$ for $s>0$ and $\vect{e}_k \in \R^{3}$ is the $k$th standard basis vector. Then, we define
\(\tensor{A}(s) := \Phi(\vect  x(s))=x(s) + \vect{e}_2^{\otimes 3}\).

For every $s = 0,1,3,5$, we created a perturbed decomposition $\vect{x}'(s) = R(\vect{x}(s), 10^{-20} \cdot \tfrac{\tensor{X}}{\Vert \tensor{X}\Vert})$, where $R$ is the aforementioned retraction and the entries of $\tensor{X}$ are chosen from the standard normal distribution. We also sampled a perturbed tensor $\tensor{A}'(s) := \tensor{A}(s) + 10^{-10} \tfrac{\tensor{Z}}{\|\tensor{Z}\|}$, where the entries of $\tensor{Z}$ are also standard normal.

For verifying the linear convergence, the RGN method was applied to $\tensor{A}'(s)$ while the quadratic convergence was checked by applying the RGN method to $\tensor{A}(s)$, both starting from $\vect{x}'(s)$. In all tested cases, the RGN method generated a sequence $\vect{x}_1(s), \vect{x}_2(s), \ldots$ in $\Var{M}$ converging to a local minimizer $\vect{x}_\star(s) \in \Var{M}$. The residual $\|F(\vect{x}_\star(s))\|$ was approximately $7 \cdot 10^{-11}$ in all cases.

The results are shown in Figures \ref{fig_convergence_plots}(a) and \ref{fig_convergence_plots}(c), illustrating respectively the predicted linear and quadratic convergence. The graphs confirm the prime message of this letter:
\emph{the convergence speed of the RGN method deteriorates when the geometric condition number increases}, as \refthm{thm_rgn_convergence} predicts.

\begin{figure}[tb!]
\centering
\includegraphics{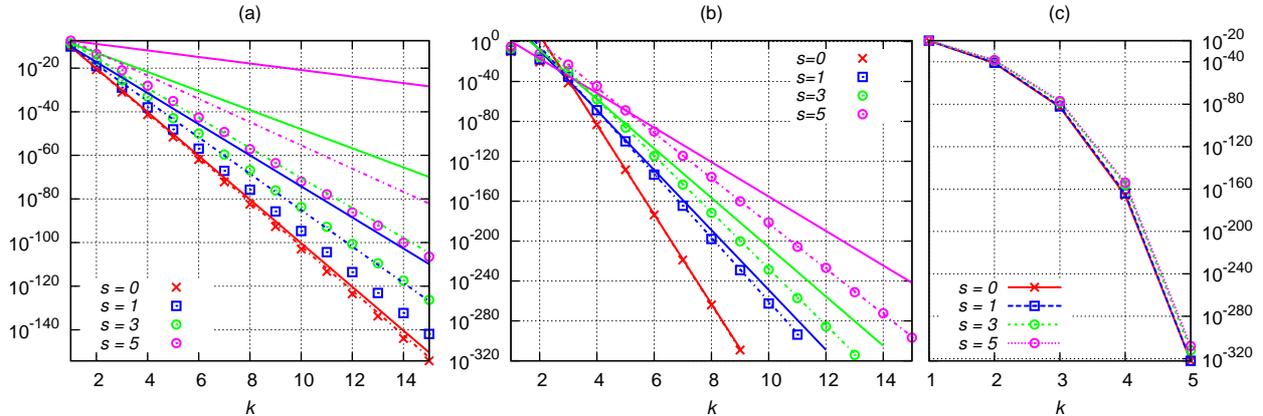}
\caption{
The data points show the distance $\| \vect{x}_k(s) - \vect{x}_\star(s) \|$ for the sequence of points $\vect{x}_k(s)$, $k=1, 2, \ldots$, computed by the RGN method in function of $k$ for $s = 0,1,3,5$. The condition numbers of the local optimizers, rounded to two significant digits, were $\kappa = 1.0 \cdot 10^0, 2.7 \cdot 10^1, 1.5\cdot 10^3,$ and $9.3 \cdot 10^4$ for respectively $s=0,1,3,$ and $5$.
In (a) and (b) linear convergence rates are illustrated when $\|F(\vect{x}_\star(s))\| \ne 0$, and (c) shows the quadratic convergence rate when the residual $\|F(\vect{x}_\star)\|$ vanishes.
Figures (a) and (c) show instances of random perturbations, while (b) employed adversarial perturbations.
In figure (b), the linear rate of convergence is not immediately observed; therefore the theoretical estimates were applied starting from the least $k$ where $\| \vect{x}_k(s) - \vect{x}_\star(s) \| \le 10^{-50}$.
In figures (a) and (b), the full lines (\protect\tikz{ \protect\draw[very thick,color=red] (0,0) -- (.7,0); \protect\draw (0,0) node {};}, \protect\tikz{ \protect\draw[very thick,color=blue] (0,0) -- (.7,0); \protect\draw (0,0) node {};}, \protect\tikz{\protect\draw[very thick,color=green] (0,0) -- (.7,0); \protect\draw (0,0) node {};}, \protect\tikz{ \protect\draw[very thick,color=magenta] (0,0) -- (.7,0); \protect\draw (0,0) node {};}) indicate the theoretical upper bounds from \refthm{thm_rgn_convergence}, i.e., $\frac{(1+\sqrt{5}) C \kappa^2}{2 \alpha}$ from \refeqn{wedin_appl}. The dashed lines (\protect\tikz{ \protect\draw[very thick,dash pattern={on 4pt off 2pt on 1pt off 2pt},color=red] (0,0) -- (0.8,0); \protect\draw (0,0) node {};}, \protect\tikz{ \protect\draw[very thick,color=blue,dash pattern={on 4pt off 2pt on 1pt off 2pt}] (0,0) -- (.8,0); \protect\draw (0,0) node {};}, \protect\tikz{\protect\draw[very thick,color=green,dash pattern={on 4pt off 2pt on 1pt off 2pt}] (0,0) -- (.8,0); \protect\draw (0,0) node {};}, \protect\tikz{ \protect\draw[very thick,dash pattern={on 4pt off 2pt on 1pt off 2pt},color=magenta] (0,0) -- (0.8,0); \protect\draw (0,0) node {};}) indicate the upper bounds obtained from \refthm{thm_rgn_convergence}, where the constants on the right-hand sides of \refeqn{wedin_appl} are \emph{estimated heuristically} as $E(s)$.}
\label{fig_convergence_plots}
\end{figure}

As the full lines in \reffig{fig_convergence_plots}(a) show, the multiplicative constants derived in \refthm{thm_rgn_convergence} can be pessimistic, especially when the condition number is large. We attribute this to bound \refeqn{wedin_appl}; while it is sharp \cite[p. 152]{MPT}, it is very pessimistic in this experiment. A qualitatively better description of the convergence is shown as the dashed lines in \reffig{fig_convergence_plots}(a), where the constant in \refeqn{wedin_appl} was estimated heuristically as $E(s) := \frac{\| J^\dagger - J_\star^\dagger \|}{\|\vect{x}_1(s) - \vect{x}_\star(s)\|}$, where $J$ is the matrix of $\deriv{F}{\vect{x}_1(s)}$ and $J_\star$ is the matrix of $\deriv{F}{\vect{x}_\star(s)}$ as in the proof of \refthm{thm_rgn_convergence}.

\subsection{Experiment 2: Adversarial perturbations.}
For illustrating the sharpness of the bound \refeqn{wedin_appl}, we performed an additional experiment with tensors in $\R^3\otimes \R^3\otimes \R^3\cong \R^{27}$. This time we constructed an adversarially perturbed starting point $\vect{x}'(s)$ by generating a random tensor {$\tensor{N} \in \R^{27}$} with entries sampled from the standard normal distribution, then computing numerically the gradient $\vect{g}$ of the function $f(\vect{x}) = \frac{1}{2}\| ((\deriv{\Phi}{\vect{x}(s)})^\dagger - (\deriv{\Phi}{\vect{x}})^\dagger) \tensor{N} \|^2$, and finally {setting} $\vect{x}'(s) = R(\vect{x}(s), 10^{-20} \vect{g})$. As adversarial perturbation of $\tensor{A}(s) = \Phi(\vect{x}(s))$, we chose $\tensor{Z} \in \R^{27}$ equal to the left singular vector $\vect{u}_{14} \in \R^{27}$ corresponding to the smallest nonzero singular value $\varsigma_{14}$ of $\deriv{\Phi}{\vect{x}'(s)}$; note that $\dim \Var M = r \cdot \dim \Var{S} = 2(1+2+2+2)=14$. As before, we set $\tensor{A}'(s) = \tensor{A}(s) + 10^{-10}\frac{\tensor{Z}}{\|\tensor{Z}\|}$.

The result of applying the RGN method to $\tensor{A}'(s)$ from starting point $\vect{x}'(s)$ is shown in \reffig{fig_convergence_plots}(b). In all cases, the method converged. The condition numbers at the local minimizers $\vect{x}_\star(s)$ are about the same as in the previous experiment: the respective relative differences were less than $10^{-2}$. The final residuals $\|F(\vect{x}_\star(s) )\|$ depended on $s$, however; they were $6.90 \cdot 10^{-46}$, $8.60 \cdot 10^{-34}$, $3.58 \cdot 10^{-31}$ and $1.10 \cdot 10^{-26}$ for respectively $s=0,1,3,$ and $5$. This is why the convergence may appear at first sight to be better than in the case of random perturbations. Nevertheless, it is observed that the theoretical estimate in \refthm{thm_rgn_convergence} is indeed much closer to the observed convergence. In fact, the bounds involving the heuristic estimate $E(s)$ are visually indistinguishable from the actual data. Note in particular for $s=0$, where $\kappa=1$, that also the theoretical convergence rate from \refthm{thm_rgn_convergence} is visually indistinguishable from the data, illustrating the sharpness of the bound in \refeqn{wedin_appl}.

\vspace{-.75em}
\section*{Acknowledgements}
We thank two anonymous reviewers for their insightful and critical remarks that improved this letter.

\vspace{-.75em}
\section*{References}
\bibliographystyle{elsarticle-harv}
\bibliography{BV}

\end{document}